\documentclass[11pt]{amsart}

\usepackage{graphicx,amsmath,mathtools,amssymb}
\usepackage{epsfig,color,hyperref,fancyhdr,geometry}
\usepackage{dsfont}
\usepackage{epstopdf}
\theoremstyle{definition}

\newtheorem{theorem}{Theorem}[section]
\newtheorem{lemma}[theorem]{Lemma}
\newtheorem{definition}[theorem]{Definition}
\newtheorem{example}[theorem]{Example}
\newtheorem{proposition}[theorem]{Proposition}
\newtheorem{remark}[theorem]{Remark}
\newtheorem{corollary}[theorem]{Corollary}
\newtheorem{conjecture}[theorem]{Conjecture}

\numberwithin{equation}{section}
\numberwithin{figure}{section}

\begin{document}

\title{Set-theoretic Yang-Baxter cohomology of cyclic biquandles}

\author{Minyi Liang}
\author{Xiao Wang}
\author{Seung Yeop Yang}

\email[Minyi Liang]{liangmymath@gmail.com}
\email[Xiao Wang]{wangxiaotop@jlu.edu.cn}
\email[Seung Yeop Yang]{seungyeop.yang@knu.ac.kr}

\address{Department of Mathematics, Jilin University, Changchun, 130012, China}
\address{Department of Mathematics, Jilin University, Changchun, 130012, China}
\address{KNU G-LAMP Project Group, KNU Institute of Basic Sciences, Department of Mathematics, Kyungpook National University, Daegu, 41566, Republic of Korea}

\begin{abstract}
We completely determine the free parts of the set-theoretic Yang-Baxter (co)homology groups of finite cyclic biquandles, along with fully computing the torsion subgroups of their 1st and 2nd homology groups.
Furthermore, we provide upper bounds for the orders of torsions in the 3rd and higher dimensional homology groups.
This work partially solves the conjecture that the normalized set-theoretic Yang-Baxter homology of cyclic biquandles satisfy $H_{n}^{NYB}(C_{m}) = \mathbb{Z}^{(m-1)^{n-1}} \oplus \mathbb{Z}_{m}$ when $n$ is odd and $H_{n}^{NYB}(C_{m}) = \mathbb{Z}^{(m-1)^{n-1}}$ when $n$ is even. In addition, we obtain cocycle representatives of a basis for the rational cohomology group of a cyclic biquandle and introduce several non-trivial torsion homology classes.
\end{abstract}



\maketitle

\section{Introduction}

The Yang-Baxter equation \cite{Bax, CYan} first appeared in the study of theoretical physics, and is closely related to various fields: quantum field theory, the theory of quantum groups, and recently quantum information science. See \cite{Dri, Fro, Kas, KL, Tur} for further details. The Jones polynomial \cite{Jon1} and the HOMFLY-PT polynomial \cite{FYHLMO, PT} were introduced as invariants for oriented links. Since it is known that they can be obtained using certain families of solutions of the Yang-Baxter equation \cite{Jon2, Tur}, the Yang-Baxter equation plays an important role in the study of low-dimensional topology and knot theory.

A homology theory of set-theoretic Yang-Baxter operators was established by J. S. Carter, M. Elhamdadi, and M. Saito \cite{CES}. Later, it was generalized for pre-Yang-Baxter operators independently by V. Lebed \cite{Leb} and J. H. Przytycki\cite{Prz}. Biquandles \cite{KR}, a generalization of quandles, are special cases of set-theoretic Yang-Baxter operators. Since it is known that their cocycles can be used to define invariants of (virtual) knots and links as a state-sum formulation, called \emph{Yang-Baxter cocycle invariants} \cite{CES}, the homology theory has been modified in various ways to define invariants of knotted objects such as knotted surfaces \cite{KKKL, NR} and handlebody-links \cite{IIKKMO}. Moreover, one can build homotopical link invariants by using geometric realizations of those homology theories. See \cite{IT, WY} for details. The homology of special families of set-theoretic Yang-Baxter operators, known as racks and quandles, has been extensively studied, see \cite{EG, GV, LN, NP, Nos, PY1, PY2, SYan}, but other than for racks and quandles little is known. Therefore, it is important for the study of invariants of knotted objects to determine set-theoretic Yang-Baxter (co)homology groups of biquandles and find explicit formulae of their cocycles.

In this article, we explore the computation of cohomology groups of set-theoretic Yang-Baxter operators, focusing on biquandles. We determine the free parts completely and estimate the torsion parts of the integral set-theoretic Yang-Baxter cohomology groups of finite cyclic biquandles. This partially proves the following conjecture:

\begin{conjecture} \cite{PVY} \label{conjecture}
For the cyclic biquandle $C_{m}$ of order $m,$\\
$$H_{n}^{NYB}(C_{m}) = \left\{
                        \begin{array}{ll}
                          \mathbb{Z}^{(m-1)^{n-1}} \oplus \mathbb{Z}_{m} & \hbox{if $n$ is odd;} \\
                          \mathbb{Z}^{(m-1)^{n-1}} & \hbox{if $n$ is even.}
                        \end{array}
                      \right.
$$
\end{conjecture}

This paper is organized as follows. In Section \ref{section2}, (co)homology theories for a set-theoretic solution of the Yang-Baxter equation are reviewed. Section \ref{section3} contains computations of the set-theoretic Yang-Baxter cohomology groups of finite cyclic biquandles. In Section \ref{section3.1}, we completely calculate the Betti numbers for the set-theoretic Yang-Baxter (co)homology, as well as for the degenerate and normalized versions, of finite cyclic biquandles. Cocycles that form a basis for the rational set-theoretic Yang-Baxter cohomology are also investigated. The torsion parts are discussed in Section \ref{section3.2}. 
We obtain the torsion subgroups of the 1st and 2nd homology groups of finite cyclic biquandles, which signifies that we have fully computed their 1st and 2nd homology groups. For 3rd homology, we compute that $H^{YB}_{3}(C_{3})=\mathbb{Z}^{9}\oplus \mathbb{Z}_{3}.$ We also establish upper bounds for the orders of torsions in their 3rd and higher-dimensional homology groups. Moreover, several non-trivial torsion homology classes are introduced.

\bigskip
\section{Preliminaries} \label{section2}

Let $X$ be a set. A function $R:X \times X \rightarrow X \times X$ is a \emph{set-theoretic solution of the Yang-Baxter equation} or a \emph{set-theoretic pre-Yang-Baxter operator} if it satisfies the following equation, called the \emph{set-theoretic Yang-Baxter equation}
$$(R \times \textrm{Id}_{X}) \circ (\textrm{Id}_{X} \times R) \circ (R \times \textrm{Id}_{X}) = (\textrm{Id}_{X} \times R) \circ (R \times \textrm{Id}_{X}) \circ (\textrm{Id}_{X} \times R).$$
A diagrammatic representation of the set-theoretic Yang-Baxter equation is shown in Figure \ref{YBE}.
The function $R$ is a \emph{set-theoretic Yang-Baxter operator} if it is invertible.

\begin{figure}[h]
\centerline{{\psfig{figure=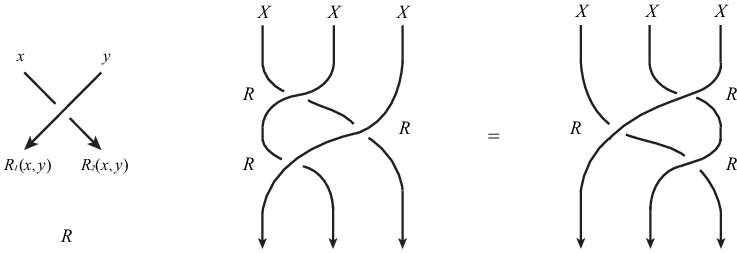,height=4.5cm}}}
\caption{An illustration of the set-theoretic Yang–Baxter equation}
\label{YBE}
\end{figure}

A set-theoretic version of the Yang-Baxter equation can be regarded as a generalization of self-distributivity which corresponds to the third Reidemeister move. There are special families of set-theoretic solutions, known as \emph{biracks} and \emph{biquandles}, that are suitable for
constructing invariants of knotted objects. Their precise definitions are as follows:

\medskip
\begin{definition}
Let $R$ be a set-theoretic Yang-Baxter operator on a given set $X,$ denoted by $$R(A_{1},A_{2})=(R_{1}(A_{1},A_{2}),R_{2}(A_{1},A_{2}))=(A_{3},A_{4}),$$ where $A_{i}\in X$ ($i=1,2,3,4$) and $R_{j}: X \times X \rightarrow X$ ($j=1,2$) are binary operations.\\
The algebraic structure $(X, R_{1}, R_{2})$ is called a \emph{birack} if it satisfies the following:
\begin{itemize}
\item For any $A_{1},A_{3}  \in X,$ there exists a unique $A_{2} \in X$ such that $R_{1}(A_{1},A_{2}) = A_{3}.$\\ In other words, $R_{1}$ is \emph{left-invertible}.

\item For any $A_{2},A_{4} \in  X,$ there exists a unique $A_{1} \in X$ such that $R_{2}(A_{1},A_{2}) = A_{4}.$\\ In other words, $R_{2}$ is \emph{right-invertible}.
\end{itemize}
A \emph{biquandle} is a birack which also satisfies the following condition:
\begin{itemize}
\item For any $A_{1} \in X,$ there is a unique $A_{2} \in X$ such that $R(A_{1},A_{2})=(A_{1},A_{2}).$\footnote{It implies that for any $A_{2} \in X,$ there is a unique $A_{1} \in X$ such that $R(A_{1},A_{2})=(A_{1},A_{2}).$}
\end{itemize}
\end{definition}
\medskip

The following are basic examples of biquandles.

\medskip
\begin{example}
\begin{enumerate}
\item Let $C_{m}$ be the cyclic group of order $m.$ The map $R:C_{m} \times C_{m} \rightarrow C_{m} \times C_{m}$ defined by $$R(a,b)=(R_{1}(a,b), R_{2}(a,b))=(b+1, a-1)$$ is a set-theoretic Yang-Baxter operator and, moreover, $(C_{m}, R_{1}, R_{2})$ forms a biquandle, called the \emph{cyclic biquandle} of order $m.$
\item \cite{CES} Let $k$ be a commutative ring with unity $1$ and with units $s$ and $t$ such that $(1-s)(1-t)=0$.
    The function $R:k \times k \rightarrow k \times k$ given by $$R(x,y)=(R_{1}(x,y), R_{2}(x,y))=((1-s)x+sy,tx+(1-t)y)$$ is a set-theoretic Yang-Baxter operator, and $(k, R_{1}, R_{2})$ forms a biquandle, called an \emph{Alexander biquandle}.\\
    For example, if we let $k=C_{m}$ with units $s$ and $t$ such that $m=|(1-s)(1-t)|,$ the function $R$ defined as above forms a set-theoretic Yang-Baxter operator and $(C_{m},R_{1},R_{2})$ is a biquandle.
\end{enumerate}
\end{example}

Let us review the cohomology theory for the set-theoretic Yang-Baxter equation based on \cite{CES, WY}. Let $R:X \times X \rightarrow X \times X$ be a set-theoretic Yang-Baxter operator on a given set $X.$ For each integer $n>0,$ let $C_{n}^{YB}(X)$ be the free abelian group generated by the elements of $X^{n},$ otherwise we let $C_{n}^{YB}(X)=0.$ For a fixed abelian group $A,$ we define the cochain complex $(C^{n}_{YB}(X;A), \delta^{n}),$ where $C^{n}_{YB}(X;A)=\text{Hom}(C_{n}^{YB}(X), A)$ with the $n$th boundary homomorphism
$$(\delta^{n}f)(x_{1},\ldots,x_{n+1})=\sum_{i=1}^{n+1} (-1)^{i-1} \Big\{ (d^{i,n+1}_{l}f)(x_{1},\ldots,x_{n+1})-(d^{i,n+1}_{r}f)(x_{1},\ldots,x_{n+1}) \Big\},$$
where $d^{i,n+1}_{l}, d^{i,n+1}_{r}$ are the dual homomorphisms of the face maps\footnote{The face maps can be illustrated more intuitively as shown in Figure \ref{YBfacemap}. See \cite{Leb, Prz} for further details.} $d_{i,n+1}^{l},d_{i,n+1}^{r}:C_{n+1}^{YB}(X) \rightarrow C_{n}^{YB}(X),$ respectively, given by
$$d_{i,n+1}^{l} = (R_{2} \times \textrm{Id}_{X}^{\times (n-1)}) \circ (\textrm{Id}_{X} \times R \times \textrm{Id}_{X}^{\times (n-2)}) \circ \cdots \circ (\textrm{Id}_{X}^{\times (i-2)} \times R \times \textrm{Id}_{X}^{\times (n-i+1)}),$$
$$d_{i,n+1}^{r} = (\textrm{Id}_{X}^{\times (n-1)} \times R_{1}) \circ (\textrm{Id}_{X}^{\times (n-2)} \times R \times \textrm{Id}_{X}) \circ \cdots \circ (\textrm{Id}_{X}^{\times (i-1)} \times R \times \textrm{Id}_{X}^{\times (n-i)}).$$

\begin{figure}[h]
\centerline{{\psfig{figure=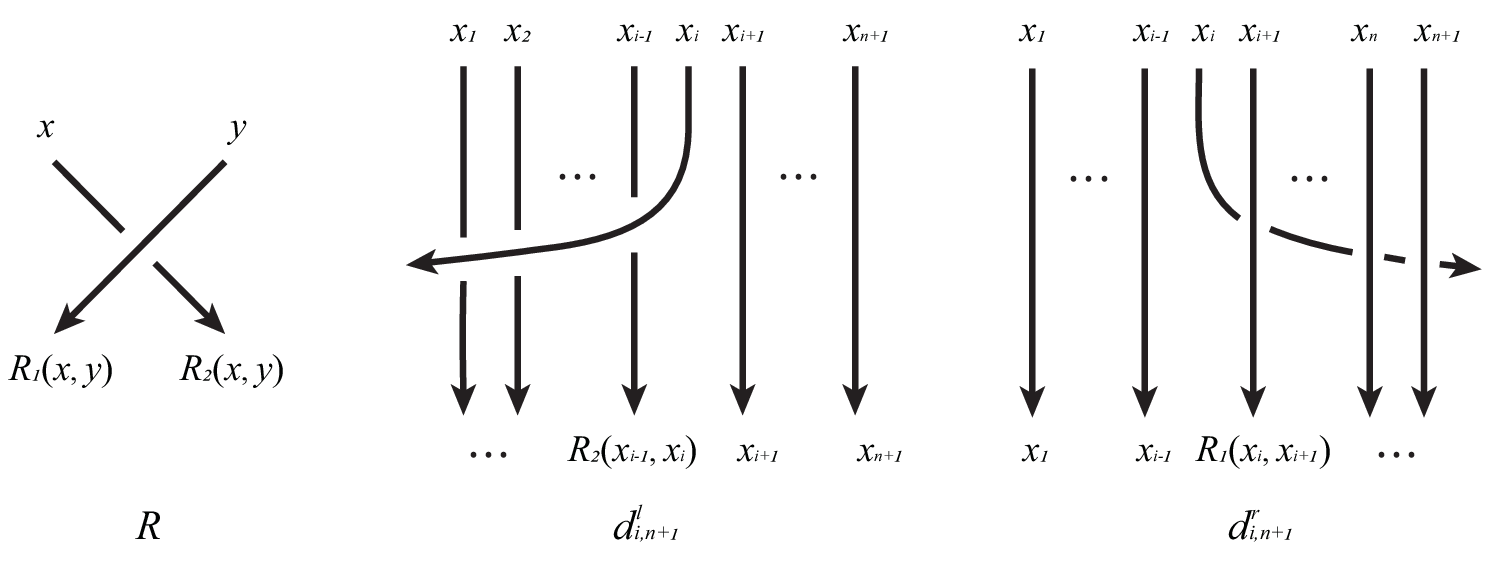,height=5.2cm}}}
\caption{A diagrammatic description of the face maps $d_{i,n+1}^{l}$ and $d_{i,n+1}^{r}$}
\label{YBfacemap}
\end{figure}

\begin{definition}
The cohomology group $H^{n}_{YB}(X;A)=H^{n}(C^{*}_{YB}(X;A))$ is called the $n$th \emph{set-theoretic Yang-Baxter cohomology group} of $X$ with coefficient in $A$ (cf. \cite{KKKL}).
\end{definition}

Let $C_{n}^{D}(X)$ be the subgroup of $C_{n}^{YB}(X)$ given by
$$C_{n}^{D}(X)=\text{Span}\{(x_{1},\ldots,x_{n})\in C_{n}^{YB}(X)~|~ R(x_{i},x_{i+1})=(x_{i},x_{i+1}) \text{~for some~} i=1,\ldots,n-1\},$$
if $n \geq 2.$ Otherwise, we let $C_{n}^{D}(X)=0.$ Then $C_{*}^{D}(X) = (C_{n}^{D}(X), \partial_{n})$ forms a chain subcomplex of the chain complex $C_{*}^{YB}(X) = (C_{n}^{YB}(X), \partial_{n}),$ where $\partial_{n}=\sum_{i=1}^{n} (-1)^{i-1}(d_{i,n}^{l} - d_{i,n}^{r}).$ Moreover, one can construct the quotient chain complex $C_{*}^{NYB}(X)=C_{*}^{YB}(X) \big/ C_{*}^{D}(X)$ consisting of $C_{n}^{NYB}(X)=C_{n}^{YB}(X) \big/ C_{n}^{D}(X)$ together with the maps $\overline{\partial}_{n} : C_{n}^{YB}(X) \big/ C_{n}^{D}(X) \rightarrow C_{n-1}^{YB}(X) \big/ C_{n-1}^{D}(X)$ uniquely determined by the universal property of the quotient. We define the cochain groups by $C^{n}_{D}(X;A)= \text{Hom}(C_{n}^{D}(X), A)$ and $C^{n}_{NYB}(X;A)= \text{Hom}(C_{n}^{NYB}(X), A).$

\begin{definition}
\begin{enumerate}
  \item The homology group $H_{n}^{YB}(X;A)=H_{n}(C_{*}^{YB}(X;A))$ is called the $n$th \emph{set-theoretic Yang-Baxter homology group} of $X$ with coefficient in $A.$
  \item The $n$th \emph{degenerate set-theoretic Yang-Baxter homology and cohomology groups} of $X$ with coefficient in $A$ are $H_{n}^{D}(X;A)=H_{n}(C_{*}^{D}(X;A))$ and $H^{n}_{D}(X;A)=H^{n}(C^{*}_{D}(X;A)).$
  \item The $n$th \emph{normalized set-theoretic Yang-Baxter homology and cohomology groups} of $X$ with coefficient in $A$ are $H_{n}^{NYB}(X;A)=H_{n}(C_{*}^{NYB}(X;A))$ and $H^{n}_{NYB}(X;A)=H^{n}(C^{*}_{NYB}(X;A)).$
\end{enumerate}
\end{definition}

\bigskip
\section{Cohomology of finite cyclic biquandles} \label{section3}

In the following subsections, we present a method of calculating the rank of the free part and study annihilation of torsion in the set-theoretic Yang-Baxter cohomology of cyclic biquandles.

\bigskip
\subsection{The Betti number of a finite cyclic biquandle} \label{section3.1}~\\

\begin{definition}
For a biquandle $X:=(X, R_{1}, R_{2}),$ the group
$$G_{X}=\langle X ~|~ xy = R_{1}(x,y)R_{2}(x,y) \text{~for all~} x,y \in X \rangle$$
is called the \emph{structure group}\footnote{The structure group of a quandle was discussed in \cite{Joy}, and the structure group of a braided set, which can be regarded as its generalization, was studied in \cite{Sol}.} of $X.$
\end{definition}

\medskip
Let $X$ be a biquandle. Then $C^{n}_{YB}(X;A)$ forms a right $G_{X}$-module, with the action defined on the generators by
$$(f\cdot y)(x_{1},\ldots,x_{n})=f\circ d_{n+1,n+1}^{l}(x_{1},\ldots,x_{n},y).$$

\medskip
\begin{proposition} \label{prop}
Suppose that a biquandle $X$ satisfies the following property:
\begin{equation} \label{property} \tag{I}
f\cdot y = f\cdot R_{1}(a,y) \text{~for all~} f \in C^{n}_{YB}(X;A) \text{~and for all~} a,y \in X.
\end{equation}
Then the coboundary homomorphism $\delta^{n}:C^{n}_{YB}(X;A) \rightarrow C^{n+1}_{YB}(X;A)$ is a $G_{X}$-module homomorphism.
\end{proposition}

\begin{proof}
Let $y$ be a generator of $G_{X},$ and let $f \in C^{n}_{YB}(X;A).$\\
We will show that $\delta^{n}(f \cdot y) = \delta^{n}(f) \cdot y.$\\
Note that
$$\delta^{n}(f \cdot y)(x_{1},\ldots,x_{n+1}) = \sum\limits_{i=1}^{n+1} (-1)^{i-1} \Big\{ ((f \cdot y) \circ d_{i,n+1}^{l}-(f \cdot y) \circ d_{i,n+1}^{r})(x_{1},\ldots,x_{n+1}) \Big\},$$
$$(\delta^{n}(f) \cdot y)(x_{1},\ldots,x_{n+1}) = \sum\limits_{i=1}^{n+1} (-1)^{i-1} \Big\{ (f \circ d_{i,n+1}^{l}-f \circ d_{i,n+1}^{r}) (d_{n+2,n+2}^{l}(x_{1},\ldots,x_{n+1},y)) \Big\}.$$
Since $X$ satisfies the set-theoretic Yang-Baxter equation and the property (\ref{property}), we have
$$(f \cdot y) \circ d_{i,n+1}^{l}(x_{1},\ldots,x_{n+1}) = (f \circ d_{i,n+1}^{l}) \circ d_{n+2,n+2}^{l}(x_{1},\ldots,x_{n+1},y) \text{~and~}$$
$$(f \cdot y) \circ d_{i,n+1}^{r}(x_{1},\ldots,x_{n+1}) = (f \circ d_{i,n+1}^{r}) \circ d_{n+2,n+2}^{l}(x_{1},\ldots,x_{n+1},y),$$ as desired. The two equations above can be more easily understood by the schematic explanations in Figure \ref{GRmodule}.
\end{proof}

\medskip
\begin{remark} \label{rem1}
All cyclic biquandles and Alexander biquandles satisfy the property (\ref{property}).
\end{remark}

\begin{figure}[h]
\centerline{{\psfig{figure=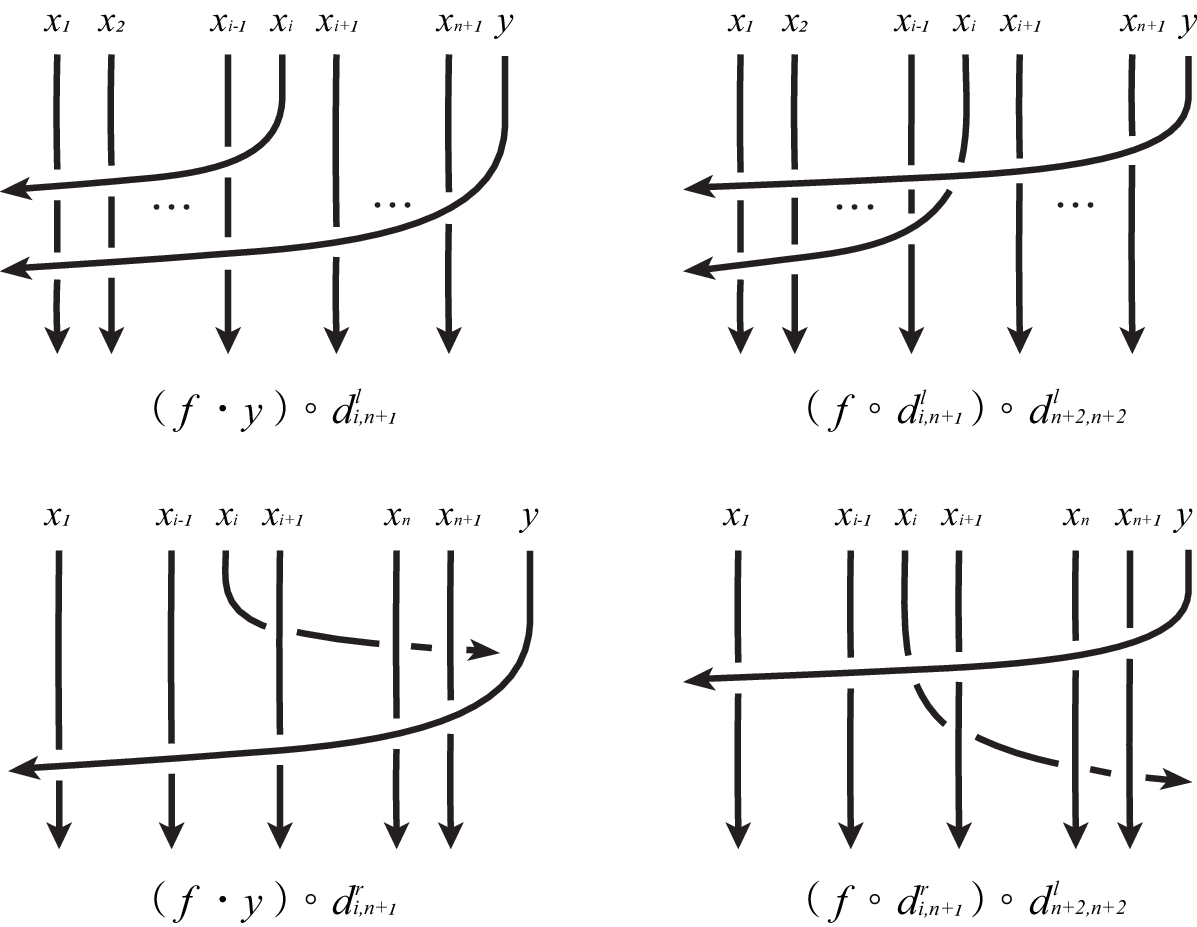,height=7.5cm}}}
\caption{$\delta^{n}(f \cdot y) = \delta^{n}(f) \cdot y$}
\label{GRmodule}
\end{figure}

\medskip

From now on, we assume that every biquandle $X$ is finite and the multiplication by $|X|$ is an automorphism of $A.$

\medskip
\begin{lemma} \label{lem}
If a biquandle $X$ satisfies the property (\ref{property}), then $[f]=[f\cdot (\frac{1}{|X|}\sum\limits_{y\in X}y)]$ in $H^{n}_{YB}(X;A).$
\end{lemma}

\begin{proof}
Let $f \in Z_{YB}^{n}(X ; A)$ be an $n$-cocycle. We consider $f'\in C^{n-1}_{YB}(X;A)$ defined by
$$f'(x_{1},\ldots,x_{n-1})=\frac{1}{|X|}\sum\limits_{y\in X}f(x_{1},\ldots,x_{n-1},y).$$
Note that $(f' \circ d_{i,n}^{r})(x_{1},\ldots,x_{n}) = \frac{1}{|X|}\sum\limits_{y\in X}(f \circ d_{i,n+1}^{r})(x_{1},\ldots,x_{n},y)$ by the left-invertibility condition of a biquandle. Then we have
\begin{flalign*}
(\delta^{n-1}f')(x_{1},\ldots,x_{n})
&=\sum\limits_{i=1}^{n} (-1)^{i-1} \Big\{ (f' \circ d_{i,n}^{l})(x_{1},\ldots,x_{n})-(f' \circ d_{i,n}^{r})(x_{1},\ldots,x_{n}) \Big\}\\
&=\sum\limits_{i=1}^{n} (-1)^{i-1} \frac{1}{|X|}\sum\limits_{y\in X} \Big\{ (f \circ d_{i,n+1}^{l})(x_{1},\ldots,x_{n},y)-(f \circ d_{i,n+1}^{r})(x_{1},\ldots,x_{n},y)\Big\}\\
&=\sum\limits_{i=1}^{n+1} (-1)^{i-1} \frac{1}{|X|}\sum\limits_{y\in X} \Big\{ (f \circ d_{i,n+1}^{l})(x_{1},\ldots,x_{n},y)-(f \circ d_{i,n+1}^{r})(x_{1},\ldots,x_{n},y)\Big\}\\
&~ -(-1)^{n}\frac{1}{|X|}\sum\limits_{y\in X} \Big\{ (f \circ d_{n+1,n+1}^{l})(x_{1},\ldots,x_{n},y)-(f \circ d_{n+1,n+1}^{r})(x_{1},\ldots,x_{n},y)\Big\}\\
&=\frac{1}{|X|}\sum\limits_{y\in X}(\delta^{n}f)(x_{1},\ldots,x_{n},y) - (-1)^{n}\frac{1}{|X|}\sum\limits_{y\in X} \Big\{(f \cdot y)(x_{1},\ldots,x_{n}) - f(x_{1},\ldots,x_{n}) \Big\}\\
&=(-1)^{n+1} \Big\{ \frac{1}{|X|}\sum\limits_{y\in X} (f \cdot y)(x_{1},\ldots,x_{n}) - f(x_{1},\ldots,x_{n}) \Big\}
\end{flalign*}
as desired.
\end{proof}
\medskip

Now let us assume that a biquandle $X$ satisfies the property (\ref{property}). Then we can consider the chain subcomplex $C^{*}_{\textrm{inv}}(X;A):=C^{*}_{YB}(X;A)^{G_{X}}$ of the cochain complex $C^{*}_{YB}(X;A)$ by Proposition \ref{prop}, and the yielded cohomology $H^{n}_{\textrm{inv}}(X;A):=H^{n}(C^{*}_{\textrm{inv}}(X;A)).$

\medskip
\begin{lemma}\label{inv}
Let $C_{m}$ be the cyclic biquandle of order $m.$ Then $H^{*}_{\textrm{inv}}(C_{m};A)=H^{*}_{YB}(C_{m};A).$
\end{lemma}

\begin{proof}
Let $g=\sum\limits_{y\in C_{m}}y.$ By Remark \ref{rem1} and Lemma \ref{lem}, $[f] = [f\cdot (\frac{g}{m})^{k}]$ in $H^{*}_{YB}(C_{m};A)$ for every positive integer $k.$
Consider $P=\frac{1}{m}\sum\limits_{i=0}^{m-1}(\frac{1}{m^{i}}g^{i}).$ Note that for any $(a,b) \in C_{m} \times C_{m},$ $R_{2}(a,b)=a-1,$ i.e., the value of $R_{2}$ does not depend on the choice of $b.$ Then for each $i$ we have $(f\cdot (\frac{1}{m^{i}} g^{i}))(x_{1},\ldots,x_{n})=f(x_{1}-i,\ldots,x_{n}-i),$ and therefore,
$$(f\cdot P)(x_{1},\ldots,x_{n}) = \frac{1}{m} \sum\limits_{i=0}^{m-1} f(x_{1}-i,\ldots,x_{n}-i).$$
Then for any $z \in C_{m},$ $(f\cdot Pz)(x_{1},\ldots,x_{n}) = \frac{1}{m} \sum\limits_{i=0}^{m-1} f(x_{1}-i-1,\ldots,x_{n}-i-1) = (f\cdot P)(x_{1},\ldots,x_{n}).$
Thus, $P$ projects the elements of $C^{*}_{YB}(C_{m};A)$ onto $C^{*}_{\textrm{inv}}(C_{m};A).$ Moreover, since $P$ commutes with the boundary homomorphisms,
$$C^{*}_{YB}(C_{m};A) = C^{*}_{\textrm{inv}}(C_{m};A) \oplus C^{*}_{YB}(C_{m};A)(1-P).$$
Note that [$f \cdot P^{2}] = [(f \cdot P) \cdot (\frac{1}{m}\sum\limits_{i=0}^{m-1}(\frac{1}{m^{i}}g^{i}))] = [\frac{1}{m}\sum\limits_{i=0}^{m-1}(f \cdot P)] = [f \cdot P]$ in $H^{*}_{YB}(C_{m};A).$\\
Now we let $[r] \in H^{n}(C^{*}_{YB}(C_{m};A)(1-P)).$ Then $[r]=[r \cdot P]$ by Lemma \ref{lem}.\\
On the other hand, there exists $s \in C^{*}_{YB}(C_{m};A)$ such that $[r] = [s - s \cdot P].$ Then
$$[r] = [r \cdot P] = [s \cdot P - s \cdot P^{2}] = [s \cdot P] - [s \cdot P] = 0.$$
Hence, $C^{*}_{YB}(C_{m};A)(1-P)$ is acyclic, as desired.
\end{proof}
\medskip

\begin{theorem}\label{m^{n-1}}
The $n$th Betti number of the cyclic biquandle $C_{m}$ of order $m$ is $m^{n-1}.$
\end{theorem}

\begin{proof}
We take $A$ to be $\mathbb{Q}.$ Then the $n$th Betti number of $C_{m},$ i.e., $\textrm{dim}H^{n}_{YB}(C_{m};\mathbb{Q}),$ is equal to $\textrm{dim}H^{n}_{\textrm{inv}}(C_{m};\mathbb{Q})$ by Lemma \ref{inv}.

Note that for each $(x_{1}, \ldots , x_{n}) \in C_{n}^{YB}(C_{m};\mathbb{Q}),$ its orbit is $\{ (x_{1}-i, \ldots , x_{n}-i) ~|~ i=0, 1, \ldots, m-1\}.$ Thus, $C^{n}_{\textrm{inv}}(C_{m};\mathbb{Q})$ has a basis consists of functions that are constant on each of these orbits.\\
If we let $f \in C^{n}_{\textrm{inv}}(C_{m};\mathbb{Q}),$ then
\begin{flalign*}
(\delta^{n}f)(x_{1},\ldots,x_{n+1})&=\sum_{i=1}^{n+1} (-1)^{i-1} \Big\{ (f \circ d_{i,n+1}^{l})(x_{1},\ldots,x_{n+1})-(f \circ d_{i,n+1}^{r})(x_{1},\ldots,x_{n+1}) \Big\}\\
&=\sum_{i=1}^{n+1} (-1)^{i-1} \Big\{ f(x_{1}-1,\ldots,x_{i-1}-1, x_{i+1}, \ldots, x_{n+1})\\
&\hspace{3cm} -f(x_{1},\ldots,x_{i-1},x_{i+1}+1, \ldots, x_{n+1}+1) \Big\}\\
&=0
\end{flalign*}
because $(x_{1}-1,\ldots,x_{i-1}-1, x_{i+1}, \ldots, x_{n+1})$ and $(x_{1},\ldots,x_{i-1},x_{i+1}+1, \ldots, x_{n+1}+1)$ are in the same orbit. That is, $H^{n}_{\textrm{inv}}(C_{m};\mathbb{Q}) = C^{n}_{\textrm{inv}}(C_{m};\mathbb{Q}),$ and therefore
$$\textrm{dim}H^{n}_{YB}(C_{m};\mathbb{Q}) = \textrm{dim}H^{n}_{\textrm{inv}}(C_{m};\mathbb{Q}) = \textrm{dim} C^{n}_{\textrm{inv}}(C_{m};\mathbb{Q}) = m^{n-1}.$$
\end{proof}

Note that $C_{n}^{D}(C_{m})=\text{Span}\{(x_{1},\ldots,x_{n})\in C_{n}^{YB}(C_{m})~|~ x_{i} = x_{i+1}+1 \text{~for some~} i=1,\ldots,n-1\},$ i.e., $\text{rank}C_{n}^{D}(C_{m}) = m^{n} - m(m-1)^{n-1}.$ In a similar way as above, one can show that $\textrm{dim}H^{n}_{D}(C_{m};\mathbb{Q}) = m^{n-1} - (m-1)^{n-1}.$ Furthermore, it was proved in \cite{PVY} that the set-theoretic Yang-Baxter homology group of $C_{m}$ splits into the normalized and degenerate parts, and by the universal coefficient theorem for cohomology we completely calculate the free parts:

\begin{corollary} \label{ranks}
For the cyclic biquandle $C_{m}$ of order $m,$
\begin{enumerate}
  \item $\text{rank}H_{n}^{YB}(C_{m}) = m^{n-1},$
  \item $\text{rank}H_{n}^{D}(C_{m}) = m^{n-1} - (m-1)^{n-1},$
  \item $\text{rank}H_{n}^{NYB}(C_{m}) = (m-1)^{n-1}.$
\end{enumerate}
\end{corollary}

The above partially proves Conjecture \ref{conjecture}.

\begin{corollary} \label{formula}
Let $\partial_{n}$ be the $n$th boundary homomorphism of the chain complex $(C_{n}^{YB}(C_m), \partial_{n}).$ Then we have
\begin{enumerate}
    \item $\textrm{dim} (\textrm{Im} \partial_{n})= \frac{(-1)^n(m-1)\{(-1)^{n}m^{n}+m\})}{m(m+1)}$
    \item $\textrm{dim} (\textrm{Ker} \partial_{n})= m^n- \frac{(-1)^n(m-1)\{(-1)^{n}m^{n}+m\})}{m(m+1)}$
\end{enumerate}
\end{corollary}

\begin{proof}
These follow directly from the recursive formulas given below, which are obtained by Corollary \ref{ranks}(1):\\
$$\left\{
                                          \begin{array}{l}
                                            \textrm{dim} (\textrm{Ker} \partial_{n}) + \textrm{dim} (\textrm{Im} \partial_{n}) = m^n ; \\
                                            \textrm{dim} (\textrm{Ker} \partial_{n}) - \textrm{dim} (\textrm{Im} \partial_{n+1}) = m^{n-1};\\
                                            \textrm{dim} (\textrm{Ker} \partial_{1})=m, ~\textrm{dim} (\textrm{Im} \partial_{1})=0.
                                          \end{array}
                                        \right.
$$
\end{proof}

Let $\mathcal{C}=(C_{n}, \partial_{n})$ be a chain complex. For each chain group $C_{n},$ we consider its dual cochain groups $C_{n}^{*}=\textrm{Hom}(C_{n}, \mathbb{Z})$ and $\widetilde{C_{n}}^{*}=\textrm{Hom}(C_{n}, \mathbb{Q}),$ and the inclusion $\iota: \mathbb{Z} \hookrightarrow \mathbb{Q}.$ For an $n$-cocycle $f \in Z_{n}^{*},$ we denote by $[f]_{\mathbb{Q}}=[\iota \circ f].$

Let $C_{m}$ be the cyclic biquandle of order $m.$ For each $n$-tuple $\textbf{x}=(x_{1}, \ldots, x_{n}) \in (C_{m})^{n},$ we consider the function $f_{\textbf{x}} \in C^{n}_{YB}(C_{m};\mathbb{Z})$ defined by
$$f_{\textbf{x}}(y_{1}, \ldots , y_{n})=\left\{
                                          \begin{array}{ll}
                                            1 & \hbox{if $(y_{1}, \ldots , y_{n})=(x_{1}-i, \ldots , x_{n}-i)$ for some $i=0, \ldots, m-1$;} \\
                                            0 & \hbox{otherwise,}
                                          \end{array}
                                        \right.
$$
and extended linearly on $C_{n}^{YB}(C_{m}).$
There are in total $m^{n-1}$ such functions. Let us denote them by $F_{1}, \ldots , F_{m^{n-1}}$ and define the set $FR^{n}(C_{m}; \mathbb{Z})=\{F_{i} ~|~ i=1,2, \ldots, m^{n-1}\}.$

\medskip
\begin{remark}
The set $\{[F_{i}]_{\mathbb{Q}} ~|~ F_{i} \in FR^{n}(C_{m}; \mathbb{Z}) \}$ is a basis for $H_{YB}^{n}(C_{m}; \mathbb{Q}),$ as described in the proof of Theorem \ref{m^{n-1}}.
\end{remark}

\begin{proposition}\label{basis}
Let $C_{m}$ be the cyclic biquandle of order $m.$ Then
\begin{enumerate}
  \item for any $f \in Z^{n}_{YB}(C_{m};\mathbb{Z})$ and $y \in C_{m},$ $m[f \cdot y] = m[f],$
  \item the submodule Span$\{[F_{i}] ~|~ F_{i} \in FR^{n}(C_{m}; \mathbb{Z}) \}$ is contained in $\text{Free}H^{n}_{YB}(C_{m};\mathbb{Z}),$
  \item the set $\{[F_{i}] ~|~ F_{i} \in FR^{n}(C_{m}; \mathbb{Z}) \}$ is linearly independent.
\end{enumerate}
\end{proposition}

\begin{proof}
(1) By Lemma \ref{lem}, we have $m[f]=[f \cdot (\sum\limits_{y\in C_{m}}y)] = [f\cdot (my)]= m[f\cdot y].$\\ \ \\
(2) We denote the submodule Span$\{[F_{i}] ~|~ F_{i} \in FR^{n}(C_{m}; \mathbb{Z}) \}$ of $H_{YB}^{n}(C_{m}; \mathbb{Z})$ by $M_{F}.$\\
Let $[f] \in M_{F} \cap \text{Tor}H_{YB}^{n}(C_{m}; \mathbb{Z}),$ i.e., $f=\sum_{i=1}^{m^{n-1}} k_{i}F_{i}$ ($k_{i} \in \mathbb{Z}$).\\
If $[f]=0,$ we are done. Assume that $[f] \neq 0.$\\
Then $[f]_{\mathbb{Q}} \neq 0$ in $H_{YB}^{n}(C_{m}; \mathbb{Q})$ because $H_{YB}^{n}(C_{m}; \mathbb{Q})=\text{Span}\{[F_{i}]_{\mathbb{Q}} ~|~ F_{i} \in FR^{n}(C_{m}; \mathbb{Z}) \}.$\\
Since $[f] \in \text{Tor}H_{YB}^{n}(C_{m}; \mathbb{Z}),$ $pf=\delta^{n-1}g$ for some $g \in C^{n-1}(C_{m}; \mathbb{Z})$ and for some integer $p>1.$ Thus, $p[f]_{\mathbb{Q}}=0$ in $H_{YB}^{n}(C_{m}; \mathbb{Q}),$ which contradicts with the fact that $[f]_{\mathbb{Q}} \neq 0.$ Therefore, $M_{F} \cap \text{Tor}H_{YB}^{n}(C_{m}; \mathbb{Z})$ is trivial, as desired.\\ \ \\
(3) Suppose that $\sum_{i=1}^{m^{n-1}} k_{i}[F_{i}]=0$ ($k_{i} \in \mathbb{Z}$). Then we have $0=\sum_{i=1}^{m^{n-1}} k_{i}[F_{i}]_{\mathbb{Q}},$ which implies that $k_{i}=0$ for all $i$ because $\{[F_{i}]_{\mathbb{Q}} ~|~ F_{i} \in FR^{n}(C_{m}; \mathbb{Z}) \}$ is a basis for $H_{YB}^{n}(C_{m}; \mathbb{Q}).$
\end{proof}

\bigskip
\subsection{The torsion of a finite cyclic biquandle}\label{section3.2}~\\

The 1st normalized set-theoretic Yang-Baxter homology of every finite cyclic biquandle was determined in \cite{PVY}.

\begin{theorem} \cite{PVY} \label{1st torsion}
Let $C_{m}$ be the cyclic biquandle of order $m.$ Then $H_{1}^{NYB}(C_{m}) = \mathbb{Z} \oplus \mathbb{Z}_{m}.$
\end{theorem}

We proceed to show that the 2nd set-theoretic Yang-Baxter homology of finite cyclic biquandles are torsion-free.

\begin{theorem}\label{2nd torsion-free}
For every finite cyclic biquandle, its 2nd set-theoretic Yang-Baxter homology group is torsion-free. Hence, $H_{2}^{YB}(C_{m}) = \mathbb{Z}^{m}.$
\end{theorem}

\begin{proof}
For the cyclic biquandle $C_{m}$ of order $m,$ $\textrm{dim} (\textrm{Im} \partial_{3}) = (m-1)^2$ by Corollary \ref{formula}. Thus, we only need to show that all of $(m-1)^2$ diagonal elements of the Smith normal form of $\partial_{3}$ are $\pm 1$s. 
To do this, we will find an $(m-1)^2 \times (m-1)^2$ minor $M$ of $\partial_{3}$ such that the determinant of $M$ is $\pm1,$ so that the $(m-1)$th determinant divisor of the Smith normal form of $\partial_{3}$ is $\pm1.$

Consider the minor $M$ of $\partial_{3}$ corresponding to the generators $\{(0, 1, 0), (0, 1, 1), \ldots , (0, 1, m-1), (0, 2, 0), \\ (0,2,1), \ldots , (0,2,m-1), (0,3,0), \ldots, (0,m-3,m-1), (0, m-2, 0), (0, m-2, 1), \ldots, (0, m-2, m-1), (0, m-1, 0)  \}$ of $C_{3}^{YB}(C_{m})$ and the generators $\{(1, 0), (1, 1), \ldots , (1, m-1), (2, 0), (2,1), \ldots , (2,m-1), (3,0), \ldots, (m-3,m-1), (m-2, 0), (m-2, 1), \ldots, (m-2, m-1), (0, m-1)  \}$ of $C_{2}^{YB}(C_{m})$. Then $M$ forms a lower triangular matrix with diagonal elements all equal to $1$, except for the last one, which is $-1.$ Because for each $1 \leq i \leq m-2$ and $1 \leq j \leq m-1,$

{\centering $\partial_{3}(0,i,j)=(i,j)-(i+1,j+1)-(m-1,j)+(0,j+1)+(m-1,i-1)-(0,i).$\par}
Note that in the image of $\partial_{3}$ above, the first term is $(i,j)$, and the first entries of all the remaining terms are either greater than $i$ or equal to $0.$ Moreover,

{\centering $\partial_{3}(0,m-1,0)=(m-1,m-2)-(0,m-1).$\par}
Therefore, the determinant of $M$ is $-1,$ as desired.
\end{proof}

\begin{remark}
By the splitness of the set-theoretic Yang-Baxter homology of cyclic biquandles \cite{PVY}, Corollary \ref{ranks}, Theorem \ref{1st torsion},  and Theorem \ref{2nd torsion-free}, we completely compute the 1st and 2nd set-theoretic Yang-Baxter homology, degenerate set-theoretic Yang-Baxter homology, and normalized set-theoretic Yang-Baxter homology of all finite cyclic biquandles as follows.\\ \ \\
For every positive integer $m,$
$$\left\{
  \begin{array}{ll}
    H_{1}^{YB}(C_{m}) = \mathbb{Z} \oplus \mathbb{Z}_{m}, & \hbox{} \\
    H_{1}^{D}(C_{m}) = 0, & \hbox{} \\
    H_{1}^{NYB}(C_{m}) = \mathbb{Z} \oplus \mathbb{Z}_{m}. & \hbox{}
  \end{array}
\right.$$

$$\left\{
  \begin{array}{ll}
    H_{2}^{YB}(C_{m}) = \mathbb{Z}^{m}, & \hbox{} \\
    H_{2}^{D}(C_{m}) = \mathbb{Z}, & \hbox{} \\
    H_{2}^{NYB}(C_{m}) = \mathbb{Z}^{(m-1)}. & \hbox{}
  \end{array}
\right.$$

\end{remark}

We move on to compute the torsion in the 3rd set-theoretic Yang-Baxter homology of $C_3.$ Before that, we investigate some non-trivial torsion homology classes.

\begin{proposition}
Let $C_{m}$ be the cyclic biquandle of order $m.$ For a fixed $y \in C_{m},$ the following is a $(2n+1)$-cycle:
$$A_{2n+1,y}:=\sum\limits_{a_{1},\ldots,a_{n} \in C_{m}}(y,a_{1},a_{1},\ldots,a_{n},a_{n}).$$
\end{proposition}
\begin{proof}
By direct computation.
\end{proof}

\begin{theorem}\cite{PVY}\label{pvy}
Let $C_{m}$ be the cyclic biquandle of order $m.$ For $n \geq 2,$ consider the maps $\theta_{n} \in C_{YB}^{n}(C_{m}; \mathbb{Z}_{2})$ defined by
$$\theta_{n}(x_{1}, \ldots, x_{n})=\prod_{i=1}^{n}x_{i} ~(\hbox{mod~} 2)$$
and extending linearly to all elements of $C^{YB}_{n}(C_{m}).$ Then $\theta_{n}$ is an $n$-cocycle when $m$ is even.
\end{theorem}

\begin{theorem}
Let $C_{m}$ be the cyclic biquandle of order $m.$ If $m=4k-2$ ($k \in \mathbb{N}$), then the homology class $[A_{2n+1,y}-A_{2n+1,y+1}]$ is non-trivial and belongs to $\hbox{Tor}H^{YB}_{2n+1}(C_{m}).$
\end{theorem}
\begin{proof}
For a fixed $y \in C_{m},$ we consider $B_{2n+2,y}:=\sum\limits_{z,a_{1}, \ldots ,a_{n} \in C_{m}}(z,y,a_{1},a_{1},...,a_{n},a_{n}).$
By direct computation, we have
$$\partial_{2n+2}(B_{2n+2,y})=m(A_{2n+1,y}-A_{2n+1,y+1}).$$
Thus, we only need to show $[A_{2n+1,y}-A_{2n+1,y+1}]$ is non-trivial.\\
Consider the cocycle $\theta_{2n+1}$ given in Theorem \ref{pvy}. The restriction $\theta_{2n+1}^{Z}:=\theta_{2n+1}|_{Z_{2n+1}^{YB}(C_{m})}$ induces the quotient homomorphism $\bar{\theta}_{2n+1}^{Z}:H^{YB}_{2n+1}(C_{m}) \rightarrow \mathbb{Z}_{2},$ where $Z_{2n+1}^{YB}(C_{m}) = \textrm{Ker} \partial_{2n+1}.$ Since $m=4k-2,$ there are $2k-1$ odd numbers in $C_{m}.$ Therefore, $\bar{\theta}_{2n+1}^{Z}([A_{2n+1,y}-A_{2n+1,y+1}]) \equiv (2k-1)^{n} \equiv 1~(\hbox{mod~} 2),$ that is, $[A_{2n+1,y}-A_{2n+1,y+1}]$ is non-trivial.
\end{proof}

We then approximate the torsions in the higher-dimensional set-theoretic Yang-Baxter cohomology of finite cyclic biquandles. The following is an annihilation theorem for those torsions.

\begin{theorem}\label{torsion}
Let $C_{m}$ be the cyclic biquandle of order $m.$ Then the torsion subgroup of $H^{n}_{YB}(C_{m};\mathbb{Z})$ is annihilated by $m$ if $m$ is odd and by $2m$ if $m$ is even.
\end{theorem}

\begin{proof}
Let $f \in Z_{YB}^{n}(C_{m} ; \mathbb{Z})$ be an $n$-cocycle, and let $y \in C_{m}$ be fixed.
Consider $f_{y} \in C_{YB}^{n-1}(C_{m} ; \mathbb{Z}),$ defined by
$$f_{y}(x_{1},\ldots,x_{n-1}) =f(x_{1},\ldots,x_{n-1},y).$$
Note that
\begin{flalign*}
\delta^{n-1}(f_{y})(x_{1},\ldots,x_{n})&=\delta^{n-1}(f_{y})(x_{1},\ldots,x_{n}) - \delta^{n}(f)(x_{1},\ldots,x_{n}, y)\\
&=\sum\limits_{i=1}^{n}(-1)^{i-1}f_{y} \circ (d_{i,n}^{l} - d_{i,n}^{r})(x_{1},\ldots,x_{n})\\
&\hspace{3cm} - \sum\limits_{i=1}^{n+1}(-1)^{i-1}f \circ (d_{i,n+1}^{l} - d_{i,n+1}^{r})(x_{1}, \ldots, x_{n}, y)\\
&=\sum\limits_{i=1}^{n}(-1)^{i} \Big\{ f_{y} \circ d_{i,n}^{r}(x_{1},\ldots,x_{n}) - f \circ d_{i,n+1}^{r}(x_{1},\ldots,x_{n}, y) \Big\}\\
&\hspace{3cm} -(-1)^{n} f \circ (d_{n+1,n+1}^{l} - d_{n+1,n+1}^{r}) (x_{1},\ldots,x_{n}, y)\\
&=\sum\limits_{i=1}^{n}(-1)^{i} \Big\{ f_{y} \circ d_{i,n}^{r}(x_{1},\ldots,x_{n}) - f_{y+1} \circ d_{i,n}^{r}(x_{1},\ldots,x_{n}) \Big\}\\
&\hspace{3cm} -(-1)^{n} \Big\{ (f \cdot y)(x_{1},\ldots,x_{n}) - f(x_{1},\ldots,x_{n}) \Big\}.
\end{flalign*}
Denote by $\Delta_{y}f := \sum\limits_{i=1}^{n}(-1)^{i} ( f_{y} \circ d_{i,n}^{r} - f_{y+1} \circ d_{i,n}^{r}).$
Then $\Delta_{y}f - (-1)^{n} (f \cdot y - f) \in \text{Im}\delta^{n-1}.$\\
Thus, we have
\begin{equation}
[f \cdot y] - [f] = [\Delta_{y}f]. \label{eq(1)}
\end{equation}
Then since $[f \cdot y^{k}] - [f \cdot y^{k-1}] = [\Delta_{y}f \cdot y^{k-1}]$ for every $k \in \mathbb{N},$ $[f \cdot y^{k}] - [f] = \sum\limits_{j=0}^{k-1}[\Delta_{y}f \cdot y^{j}].$
Therefore,
\begin{equation}
\Big[ f \cdot (\sum\limits_{i=1}^{m}y^{i} ) \Big] - [mf] = \sum\limits_{i=1}^{m}([f \cdot y^{i}] - [f]) = \Big[ \Delta_{y}f \cdot (\sum\limits_{i=1}^{m}\sum\limits_{j=0}^{i-1} y^{j}) \Big]. \label{eq(2)}
\end{equation}

Note that
\begin{flalign*}
f_{y} \circ d_{i,n}^{r}(x_{1},\ldots,x_{n}) &= f(x_{1},\ldots,x_{i-1}, x_{i+1}+1, \ldots , x_{n}+1, y)\\
&= f_{y} \circ d_{i,n}^{l}(x_{1}+1,\ldots,x_{n}+1)\\
&= ((f_{y} \circ d_{i,n}^{l}) \cdot y^{-1})(x_{1},\ldots,x_{n}),
\end{flalign*}
i.e., $(f_{y} \circ d_{i}^{r}) \cdot y = f_{y} \circ d_{i}^{l}.$ Then we have
\begin{flalign*}
\delta^{n-1}(f_{y}-f_{y+1}) &=\sum\limits_{i=1}^{n}(-1)^{i-1} \Big\{ (f_{y} \circ d_{i,n}^{l} - f_{y} \circ d_{i,n}^{r}) - (f_{y+1} \circ d_{i,n}^{l} - f_{y+1} \circ d_{i,n}^{r}) \Big\} \\
&=\sum\limits_{i=1}^{n}(-1)^{i-1} \Big\{ (f_{y} \circ d_{i,n}^{l} - f_{y+1} \circ d_{i,n}^{l}) - (f_{y} \circ d_{i,n}^{r} - f_{y+1} \circ d_{i,n}^{r}) \Big\} \\
&=\sum\limits_{i=1}^{n}(-1)^{i-1} \Big\{ ((f_{y} \circ d_{i,n}^{r}) \cdot y - (f_{y+1} \circ d_{i,n}^{r}) \cdot y) - (f_{y} \circ d_{i,n}^{r} - f_{y+1} \circ d_{i,n}^{r}) \Big\} \\
&= - (\Delta_{y}f) \cdot y + \Delta_{y}f.
\end{flalign*}
Then $[(\Delta_{y}f) \cdot y] = [\Delta_{y}f],$ and therefore, for every $k \in \mathbb{N},$
\begin{equation}
[\Delta_{y}f] = [(\Delta_{y}f) \cdot y] = \cdots = [(\Delta_{y}f) \cdot y^{k}]. \label{eq(3)}
\end{equation}

By the equations \ref{eq(2)} and \ref{eq(3)}, we have
\begin{equation}
\Big[f \cdot (\sum\limits_{i=1}^{m}y^{i}) \Big] - [mf] = \Big[\frac{m(m+1)}{2} \Delta_{y}f \Big]. \label{eq(4)}
\end{equation}
Moreover, $m[\Delta_{y}f] = m[f \cdot y] - m[f] = 0$ by the equation \ref{eq(1)} and Lemma \ref{basis}(1).\\
Thus, by the equation \ref{eq(4)}
$$\left\{
  \begin{array}{ll}
    m[f] = \Big[f \cdot (\sum\limits_{i=1}^{m}y^{i}) \Big], & \hbox{if $m$ is odd;} \\
    2m[f] = 2\Big[f \cdot (\sum\limits_{i=1}^{m}y^{i}) \Big], & \hbox{if $m$ is even.}
  \end{array}
\right.$$
Note that $\Big[f \cdot (\sum\limits_{i=1}^{m}y^{i}) \Big]$ belongs to the free part of $H_{YB}^{n}(C_{m} ; \mathbb{Z})$ by Lemma \ref{basis}(2).\\
Therefore, $m \text{Tor}H_{YB}^{n}(C_{m} ; \mathbb{Z})=0$ if $m$ is odd and $2m \text{Tor}H_{YB}^{n}(C_{m} ; \mathbb{Z})=0$ if $m$ is even.
\end{proof}

\medskip
\begin{remark}
The same results can be obtained for the degenerate and the normalized set-theoretic Yang-Baxter (co)homology of a finite cyclic biquandle as in Theorem \ref{torsion}, because the set-theoretic Yang-Baxter homology splits into the degenerate part and the normalized part \cite{PVY}.
\end{remark}

Combining the results above, one can compute the torsion in the 3rd set-theoretic Yang-Baxter homology of $C_3.$

\begin{theorem}
$H^{YB}_{3}(C_{3})=\mathbb{Z}^{9}\oplus \mathbb{Z}_{3}$.
\end{theorem}

\begin{proof}
Note that $\textrm{Free} H^{YB}_{3}(C_{3})=\mathbb{Z}^{9}$ by Corollary \ref{ranks}(1).
For $\partial_{4}: C_{4}^{YB}(C_{3}) \rightarrow C_{3}^{YB}(C_{3}),$ $\textrm{dim} (\textrm{Im} \partial_{4}) = 14$ by Corollary \ref{formula}.
Consider the $13 \times 13$ minor $M$ of $\partial_{4}$ corresponding to the generators $\{(0,1,0,0), \\(0,1,0,1), (0,1,0,2), (0,1,1,0), (0,1,1,1), (0,1,1,2), (0,1,2,0), (0,1,2,1), (0,1,2,2), (0,0,2,2)+(0,1,0,0), \\(0,0,2,0)+(0,1,0,1), (0,0,2,1)+(0,1,0,2), (0,0,1,1)+(0,1,2,2)-(0,0,2,2)-(0,1,0,0) \}$ of $C_{4}^{YB}(C_{3})$ and the generators $\{(1,0,0), (1,0,1), (1,0,2), (1,1,0), (1,1,1), (1,1,2), (1,2,0), (1,2,1), (1,2,2), (0,2,2), \\(0,2,0), (0,2,1), (0,1,1)  \}$ of $C_{3}^{YB}(C_{3})$ given below.
$$
M=\begin{bmatrix}
1 & 0 & 0 & 0 & 0 & 0 & 0 & 0 & 0 & 0 & 0 & 0 & 0 \\
0 & 1 & 0 & 0 & 0 & 0 & 0 & 0 & 0 & 0 & 0 & 0 & 0 \\
0 & 0 & 1 & 0 & 0 & 0 & 0 & 0 & 0 & 0 & 0 & 0 & 0 \\
0 & 0 & 0 & 1 & 0 & 0 & 0 & 0 & 0 & 0 & 0 & 0 & 0 \\
0 & 0 & 0 & 0 & 1 & 0 & 0 & 0 & 0 & 0 & 0 & 0 & 0 \\
0 & 0 & 0 & 0 & 0 & 1 & 0 & 0 & 0 & 0 & 0 & 0 & 0 \\
0 & 0 & 0 & 0 & 0 & 0 & 1 & 0 & 0 & 0 & 0 & 0 & 0 \\
0 & 0 & 0 & 0 & 0 & 0 & 0 & 1 & 0 & 0 & 0 & 0 & 0 \\
0 & 0 & 0 & 0 & 0 & 0 & 0 & 0 & 1 & 0 & 0 & 0 & 0 \\
0 & 0 & 0 & 0 & 1 & 0 & 0 & 0 & 0 & 1 & 0 & 0 & 0 \\
0 & 0 & 0 & 0 & 0 & 1 & 0 & 0 & 0 & 0 & 1 & 0 & 0 \\
0 & 0 & 0 & 1 & 0 & 0 & 0 & 0 & 0 & 0 & 0 & 1 & 0 \\
0 & 0 & 0 & 0 & 1 & 1 & -1 & 0 & 0 & 0 & 0 & 0 & 1 
\end{bmatrix}
$$
Since the determinant of $M$ is $1,$ the $13$th determinant divisor of the Smith normal form of $\partial_{4}$ is $1.$ Therefore, if $H^{YB}_{3}(C_{3})$ is not torsion-free, then $\textrm{Tor} H^{YB}_{3}(C_{3})=\mathbb{Z}_{3}$ by the universal coefficient theorem for cohomology and Theorem \ref{torsion}.

Consider the $3$-cocycle $\phi \in C_{YB}^{3}(C_3;\mathbb{Z}_{3})$ defined by
$$\phi(x_{1},x_{2},x_{3}) = \begin{cases} 
	1 & \text{if } (x_{1},x_{2},x_{3}) = (0,1,1), (2,2,0) ;\\ 
  	2 & \text{if } (x_{1},x_{2},x_{3}) = (0,1,2), (2,2,2) ;\\ 
        0 & \text{otherwise.} 
     \end{cases}$$
For the induced homomorphism $\bar{\phi}:H_{3}^{YB}(C_3)\rightarrow \mathbb{Z}_{3}$ given by $\bar{\phi}([a])=\phi(a),$ we have $\bar{\phi}([A_{3,1}-A_{3,2}])=\phi(A_{3,1}-A_{3,2})=1 \neq 0.$ That is, $[A_{3,1}-A_{3,2}] \neq 0$ in $H_{3}^{YB}(C_3).$ Moreover, $[A_{3,1}-A_{3,2}]$ is a torsion element as $\partial_{4}(B_{4,1})=3(A_{3,1}-A_{3,2}),$ where $B_{4,1}=\sum\limits_{z,a \in C_{3}}(z,1,a,a).$ Hence, $H^{YB}_{3}(C_{3})$ is not torsion-free, as desired.
\end{proof}

Table \ref{table:comp} summarizes the results of the set-theoretic Yang-Baxter cohomology of finite cyclic biquandles obtained so far.

\begin{table}[h]
\centering
\caption{Set-theoretic Yang-Baxter homology groups of cyclic biquandles}\label{table:comp}
\begin{tabular}{@{}l|llllll@{}}
 \hline
    $n$&$1$&$2$&$3$&$4$&$5$&$\cdots$\\
    \hline

    $H_n^{YB}(C_{2})$  &  $\mathbb{Z} \oplus \mathbb{Z}_{2}$  &  $\mathbb{Z}^{2}$  &  $\mathbb{Z}^{4} \oplus \mathbb{Z}_{2}$  &  $\mathbb{Z}^{8}$  &  $\mathbb{Z}^{16} \oplus \mathbb{Z}_{2}$ & \\
    $H_n^{D}(C_{2})$  &  $0$  &  $\mathbb{Z}$  &  $\mathbb{Z}^{3}$  &  $\mathbb{Z}^{7}$  &  $\mathbb{Z}^{15}$ & $\cdots$\\
    $H_n^{NYB}(C_{2})$  &  $\mathbb{Z} \oplus \mathbb{Z}_{2}$  &  $\mathbb{Z}$  &  $\mathbb{Z} \oplus \mathbb{Z}_{2}$  &  $\mathbb{Z}$  &  $\mathbb{Z} \oplus \mathbb{Z}_{2}$ &\\

    \hline

    $H_n^{YB}(C_{3})$  &  $\mathbb{Z} \oplus \mathbb{Z}_{3}$  &  $\mathbb{Z}^{3}$  &  $\mathbb{Z}^{9} \oplus \mathbb{Z}_{3}$  &  $\mathbb{Z}^{27} \oplus \textcolor{red}{\,?}$  &  $\mathbb{Z}^{81} \oplus \textcolor{red}{\,?}$ & \\
    $H_n^{D}(C_{3})$  &  $0$  &  $\mathbb{Z}$  &  $\mathbb{Z}^{5} \oplus \textcolor{red}{\,?}$  &  $\mathbb{Z}^{19} \oplus \textcolor{red}{\,?}$  &  $\mathbb{Z}^{65} \oplus \textcolor{red}{\,?}$ & $\cdots$\\
    $H_n^{NYB}(C_{3})$  &  $\mathbb{Z} \oplus \mathbb{Z}_{3}$  &  $\mathbb{Z}^{2}$  &  $\mathbb{Z}^{4}  \oplus \textcolor{red}{\,?}$  &  $\mathbb{Z}^{8} \oplus \textcolor{red}{\,?}$  &  $\mathbb{Z}^{16}  \oplus \textcolor{red}{\,?}$ & \\

    \hline

    $H_n^{YB}(C_{4})$  &  $\mathbb{Z} \oplus \mathbb{Z}_{4}$  &  $\mathbb{Z}^{4}$  &  $\mathbb{Z}^{16}  \oplus \textcolor{red}{\,?}$  &  $\mathbb{Z}^{64} \oplus \textcolor{red}{\,?}$  &  $\mathbb{Z}^{256}  \oplus \textcolor{red}{\,?}$ & \\
    $H_n^{D}(C_{4})$  &  $0$  &  $\mathbb{Z}$  &  $\mathbb{Z}^{7} \oplus \textcolor{red}{\,?}$  &  $\mathbb{Z}^{37} \oplus \textcolor{red}{\,?}$  &  $\mathbb{Z}^{175} \oplus \textcolor{red}{\,?}$ & $\cdots$\\
    $H_n^{NYB}(C_{4})$  &  $\mathbb{Z} \oplus \mathbb{Z}_{4}$  &  $\mathbb{Z}^{3}$  &  $\mathbb{Z}^{9}  \oplus \textcolor{red}{\,?}$  &  $\mathbb{Z}^{27} \oplus \textcolor{red}{\,?}$  &  $\mathbb{Z}^{81}  \oplus \textcolor{red}{\,?}$ & \\

    \hline

    $H_n^{YB}(C_{5})$  &  $\mathbb{Z} \oplus \mathbb{Z}_{5}$  &  $\mathbb{Z}^{5}$  &  $\mathbb{Z}^{25}  \oplus \textcolor{red}{\,?}$  &  $\mathbb{Z}^{125} \oplus \textcolor{red}{\,?}$  &  $\mathbb{Z}^{625}  \oplus \textcolor{red}{\,?}$ & \\
    $H_n^{D}(C_{5})$  &  $0$  &  $\mathbb{Z}$  &  $\mathbb{Z}^{9} \oplus \textcolor{red}{\,?}$  &  $\mathbb{Z}^{61} \oplus \textcolor{red}{\,?}$  &  $\mathbb{Z}^{369} \oplus \textcolor{red}{\,?}$ & $\cdots$\\
    $H_n^{NYB}(C_{5})$  &  $\mathbb{Z} \oplus \mathbb{Z}_{5}$  &  $\mathbb{Z}^{4}$  &  $\mathbb{Z}^{16}  \oplus \textcolor{red}{\,?}$  &  $\mathbb{Z}^{64} \oplus \textcolor{red}{\,?}$  &  $\mathbb{Z}^{256} \oplus \textcolor{red}{\,?}$ & \\

    \hline

    $\qquad\vdots$  &  $\quad\,\vdots$  &  $\,\vdots$  &  $\qquad\vdots$  &  $\qquad\vdots$  &  $\qquad\;\vdots$ & \\

\end{tabular}
\end{table}

\bigskip

\section{Future research}

Based on Conjecture \ref{conjecture}, we can observe that Theorem \ref{torsion} closely approximates the torsions of the set-theoretic Yang-Baxter (co)homology of finite cyclic biquandles $C_{m}$ when $m$ is odd, but not when $m$ is even. Furthermore, not much is known about the (co)homology groups of other biquandles, such as Alexander biquandles.

\bigskip
\section{Acknowledgements}
The work of Xiao Wang was supported by the National Natural Science Foundation of China (No. 11901229, No. 22341304, No. 12371029). The work of Seung Yeop Yang was supported by the National Research Foundation of Korea (NRF) grant funded by the Korean government (MSIT)(No. 2022R1A5A1033624) and by Global - Learning \& Academic research institution for Master’s·PhD students, and Postdocs (G-LAMP) Program of the National Research Foundation of Korea (NRF) grant funded by the Ministry of Education (No. RS-2023-00301914).

\end{document}